\documentclass[12pt]{amsart}
\usepackage{amssymb}

\input xy
\xyoption{all}  

\theoremstyle{plain}
\newtheorem{prop}{Proposition}
\newtheorem{thm}[prop]{Theorem}
\newtheorem{coro}[prop]{Corollary}
\newtheorem{lemm}[prop]{Lemma}
\newtheorem{conj}[prop]{Conjecture}
\newtheorem{prob}[prop]{Problem}

\newtheorem*{thm-int}{Theorem}
\theoremstyle{remark}
\newtheorem{exam}[prop]{Example}
\newtheorem{rema}[prop]{Remark}

\theoremstyle{definition}
\newtheorem{defn}[prop]{Definition}

\def\NS{{\rm NS}}

\def\Rat{{\rm Rat}}
\def\Unirat{{\rm Unirat}}

\def\Aut{{\rm Aut}}

\def\no{\noindent}

\newcommand{\Ker}{{\rm Ker}}

\def\lra{\longrightarrow}
\def\ra{\rightarrow}

\def\A{{\mathbb A}}

\def\F{{\mathbb F}}

\def\P{{\mathbb P}}
\def\Q{{\mathbb Q}}

\def\Z{{\mathbb Z}}

\def\N{{\mathbb N}}

\def\Br{{\rm Br}}
\def\SL{{\rm SL}}
\def\GL{{\rm GL}}
\def\PGL{{\rm PGL}}

\def\Bir{{\rm Bir}}
\def\Aut{{\rm Aut}}

\def\Syl{{\rm Syl}}

\begin{document}
\title[Noether's problem]{Noether's problem and descent}

\author{Fedor Bogomolov}
\address{Courant Institute of Mathematical Sciences, N.Y.U. \\
 251 Mercer str. \\
 New York, NY 10012, U.S.A.}
\address{National Research University Higher School of Economics, 
Russian Federation \\
AG Laboratory, HSE \\ 
7 Vavilova str., Moscow, Russia, 117312}

\email{bogomolo@cims.nyu.edu}

\author{Yuri Tschinkel}
\address{Courant Institute of Mathematical Sciences, N.Y.U. \\
 251 Mercer str. \\
 New York, NY 10012, U.S.A.}
\address{Simons Foundation\\
160 Fifth Av. \\
New York, NY 10010, U.S.A.}
\email{tschinkel@cims.nyu.edu}        

\keywords{Rationality, fibrations}

\begin{abstract}
We study Noether's problem from the perspective of torsors under linear algebraic groups and descent. 
\end{abstract}

\maketitle

\setcounter{section}{0}       
\section*{Introduction}
\label{sect:introduction}

This note is inspired by the theory of universal torsors developed by Colliot-Th\'el\`ene and Sansuc in connection with 
arithmetic problems on del Pezzo surfaces \cite{CTS}. This theory associates to a geometrically rational surface $X$ over a field $k$, 
with $X(k)\neq \emptyset$, a torsor 
\begin{equation}
\label{eqn:tors}
\pi:\mathcal T\stackrel{T}{\lra} X,
\end{equation}
where $T$ is the N\'eron-Severi torus of $X$, i.e., the character group of $T_{\bar{k}}$ is isomorphic, as a Galois module,  
to the geometric N\'eron-Severi lattice $\NS(X_{\bar{k}})$ of $X$. 
The torsor is viewed as a {\em descent} variety. Basic arithmetic problems on $X$, 
such as the Hasse Principle and Weak Approximation, are reduced 
to the following geometric conjecture: $\mathcal T$ is {\em rational} over $k$. 
The gist of this conjecture is that on $\mathcal T$, the arithmetic complexity of $X$ 
is {\em untwisted}: while $X$ may have nontrivial (algebraic) Brauer group,
it is eliminated via passage to a universal torsor. 
The conjecture implies in particular that $X$ is unirational over $k$, which is an open problem
for del Pezzo surfaces of degree 1.

In higher dimensions, unirational varieties may have nontrivial transcendental Brauer group, 
or more generally, nontrivial higher 
unramified cohomology.   
Examples are conic bundles over rational surfaces considered by Artin--Mumford \cite{AM}, quadric bundles \cite{CTO}, or 
Brauer--Severi bundles. 
Our motivation was to understand whether or not these obstructions to 
stable rationality can be untwisted via passage to fibrations as in \eqref{eqn:tors}. 

\begin{defn} 
\label{defn:tower}
A variety $X$ over a field $k$ 
admits a {\em rational tower} if there exists a
sequence of dominant rational maps
\begin{equation}
\label{eqn:tower}
X_n \stackrel{\pi_n}{\lra} X_{n-1} \ra  \cdots \ra X_1    \stackrel{\pi_1}{\lra} X_0:=X,
\end{equation}
over $k$, such that 
\begin{itemize}
\item[(1)] the {\em source} of the tower, $X_n$, is rational over $k$,
\item[(2)] the generic fiber of $\pi_i$ is geometrically rational and irreducible, over the function field $k(X_{i-1})$, for all $i$.   
\end{itemize}
We say that $X$ admits a {\em toric rational tower}, if in addition
\begin{itemize}
\item[(2')]   the generic fiber of $\pi_i$ is birational to a torsor under an algebraic torus, over the function field $k(X_{i-1})$, for all $i$.   
\end{itemize}
\end{defn}

\begin{conj}
\label{conj:main}
Let $X$ be a unirational variety over $k$. Then 
$X$ admits a rational tower.
\end{conj} 

This conjecture is motivated by topology: every continuous map is homotopy equivalent to a fibration with constant fiber. 
In the algebraic-geometric context, the conjecture implies in particular that {\em unramified cohomology} of the source of the tower, $X_n$, in Definition \ref{defn:tower}, is trivial. Of course, one can trivialize a given 
unramified cohomology class (with finite coefficients) 
via passage to a tower with geometrically rational fibers, as in \eqref{eqn:tower}.  
Indeed, let $K=k(X)$ be the function field of an algebraic variety over a field $k$, which contains all roots of unity.  
By the Bloch-Kato conjecture, proved by Voevodsky, the subring 
$$
\oplus_{i\ge 2} H^i(K)\subset H^*(K)
$$
of Galois cohomology of $K$, with finite coefficients, equals the ideal generated by $H^2(K)$. 
Since elements of $H^2$ are trivialized on the corresponding Brauer-Severi variety, 
we can trivialize an arbitrary finite set of 
Galois cohomology classes via passage to a tower as in \eqref{eqn:tower}:
$$
Y_m\ra Y_{m-1} \ra \cdots \ra Y_1\ra Y_0=X,
$$
where the generic fibers are forms of projective spaces. 
However,  a direct application of this construction 
does not imply Conjecture~\ref{conj:main}, as we cannot assert the rationality of $Y_m$. 
In fact, new unramified classes
may appear in the process, and we cannot even ensure that $Y_m$
has trivial unramified cohomology.

To motivate Conjecture~\ref{conj:main} from the perspective of algebra, consider the following
important class of unirational varieties: quotients $V/G$, where
$G$ is a finite group and $V$ a (finite-dimensional) faithful linear representation of $G$ over $k$. 
In \cite{BT-univer} we showed that, over $k=\bar{\F}_p$,  these 
varieties are {\em universal} for unramified cohomology:
every unramified cohomology class of an algebraic variety over $k$ is
induced from $V/G^c$, where $G^c$ is a central extension of an abelian group (see Section~\ref{sect:background} for 
more details). Thus we expect that such quotients are universal from the birational point of view as well. 
For groups of this type, and a wide class of solvable groups $G$, we prove:
\begin{itemize}
\item $V/G$ admits a toric rational tower.
\end{itemize}
Its source $X_n$ is a $k$-rational variety, untwisting all unramified cohomology of $V/G$.

\

\no
{\bf Acknowledgments.} 
We are thankful to B. Bloom-Smith for useful comments. 
This work was partially supported by 
Laboratory of Mirror Symmetry NRU HSE, RF grant 14.641.31.0001.
The first author was funded by
the Russian Academic Excellence Project `5-100'.
The first author was also supported by a Simons Fellowship 
and by the EPSRC program grant EP/M024830. 
The second author was partially supported by the NSF grant  1601912.

\section{Background}
\label{sect:background}

Let $X$ be a smooth projective geometrically rational variety over a field $k$. Among obstructions to $k$-rationality 
are the absence of $k$-rational points or the nontriviality of the {\em algebraic} Brauer group
$$
\Br(X)/\Br(k).
$$
More generally, we may consider {\em unramified cohomology} 
$$
H^i_{nr}(K,\mu_m^{\otimes  i-1}),     \quad H^i_{nr}(K,\Q/\Z(i-1)),   \quad m\in \N,
$$
(see \cite{Bog87}, \cite{CTO} for definitions and basic properties). 
These groups are birational invariants which vanish when $i > \dim(X)$ or when $X$ is 
$k$-rational. 
For $X$ smooth and projective, we have
$$
\Br(X)[m]=H^2_{nr}(K, \mu_m), 
$$
where $\mu_m$ stands for $m$-th roots of 1.  
Below, we omit the coefficients, when they are clear from the context. 
 
\begin{conj}
\label{conj:bounded}
Let $X$ be a unirational variety over an algebraically closed field of characteristic zero. 
Then its unramified cohomology is finite.
\end{conj}

Finiteness of $H^i_{nr}(X,\Q/\Z(i-1))$, for unirational $X$, is known for:
\begin{itemize}
\item $i=2$, classically,  
\item $i=3$ \cite[Proposition 3.2]{CT-Kahn}.
\end{itemize}

When $X(k)\neq \emptyset$, the theory of Colliot-Th\'el\`ene--Sansuc
provides a fibration 
$$
\pi: \mathcal T\stackrel{T}{\lra} X
$$ 
as in \eqref{eqn:tors}, such that
\begin{itemize}
\item the generic fiber of $\pi$ is a principal homogeneous space under an algebraic torus $T$ and 
\item the total space $\mathcal T$ is, {\em conjecturally}, a rational variety \cite[Conjecture H1, Section 2.8]{CTS}. 
\end{itemize}
Special cases of this are known, e.g., 
when $X$ is admits a conic bundle over $k$, with 4 degenerate fibers \cite{CT-Sko}. As evidence for 
this rationality conjecture, one has the  
proof that, over $k$ of characteristic zero, $H^2_{nr}(\mathcal T)$ is trivial \cite[Theorem 2.1.2]{CTS}. A more recent consistency check is \cite{cao}, 
where it is shown that 
\begin{itemize}
\item when $X$ is a del Pezzo surface of degree $\ge 2$, the  nontrivial part of $H^3_{nr}(\mathcal T)$ is finite and 2-primary. 
\end{itemize}

We now turn to unirational varieties which are not necessarily geometrically rational, which may have nontrivial {\em transcendental} Brauer group. 
Examples appeared in the context of {\em Noether's problem}: 

\begin{prob}[Noether]
Let $V$ be a faithful representation of a finite group $G$ over $k$. Is  $V/G$ (stably) rational? 
\end{prob}

Noether's problem has a negative solution: counterexamples are based on explicit computations of unramified cohomology, which is 
essentially a combinatorial problem, in terms of the structure of Sylow subgroups of $G$.  
Groups  $G$ with nontrivial $H^2_{nr}(V/G)$ were constructed in 
\cite{saltman}, \cite{Bog-Izv-87}, and \cite{mor}; with trivial $H^2_{nr}(V/G)$ but nontrivial $H^3_{nr}(V/G)$ in \cite{peyre} and \cite{hoshi}.

On the other hand, unramified cohomology is trivial for actions of many classical finite groups.  
In some, but not all, of these cases, we have proofs of (stable) rationality (see, e.g., \cite{B-stable}, \cite{BBB}). 

Our goal in the following sections will be to {\em untwist} unramified cohomology of $V/G$ by constructing 
rational towers as in Definition~\ref{defn:tower}. Of particular importance are solvable groups $G$. Indeed, in  
\cite{BT-univer} we proved a {\em universality} result:

\begin{thm}
\label{thm:univer}
Let $X$ be an algebraic variety of dimension $\ge 2$ over $k=\bar{\mathbb F}_p$, $\ell\neq p$ a prime,  
and 
$$
\alpha\in H^i_{nr}(k(X), \Z/\ell^n)
$$ 
an unramified class. 
Then there exist finite-dimensional $k$-vector spaces $V_j, j\in J$, depending on $\alpha$, 
such that $\alpha$ is induced, via a rational map, from an unramified class 
in the cohomology of the quotient of 
$$
\mathbb P:=\prod_{j\in J} \mathbb P(V_j)
$$ 
by a finite abelian $\ell$-group $G^a$, 
acting projectively on each factor.
\end{thm}

In other words, central extensions of abelian groups capture all unramified cohomology invariants.

\section{First properties}
\label{sect:first}

Throughout, $G$ is a linear algebraic group over a field $k$ and $V$ a finite-dimensional linear faithful representation of $G$ over $k$. 

Let $\Bir(k)$ be the set algebraic varieties over $k$, 
modulo $k$-birational equivalence, which we denote by $\sim_k$.  Let 
$$
\Rat(k)\subset \mathcal L\mathcal Q(k)\subset \mathcal G\mathcal Q(k)\subset \Unirat(k),
$$
be the classes of algebraic varieties over $k$ which are 
\begin{itemize}
\item $k$-rational, 
\item $k$-birational to $V/G$  (Linear Quotients), 
\item $k$-birational to $X/G$, where $X$ is a $k$-rational 
algebraic variety and $G$ is a subgroup of the group $\Bir\Aut(X)$ of $k$-birational automorphisms of $X$ (General Quotients), 
\item $k$-unirational, 
\end{itemize}
respectively. 
Our goal is to connect these classes, via passage to fibrations with geometrically rational generic fibers as in Definition~\ref{defn:tower}. 
We start with simple examples:

\begin{exam}
\label{exam:abelian}
Let $G$ be a finite abelian group, or an extension of a finite cyclic group by a finite abelian group. 
If $k$ is of characteristic coprime to $n:=|G|$ and 
contains $n$-th roots of 1 then $V/G$ is rational (see \cite[Theorem 6.1]{swan}). 

This can fail over nonclosed fields, even when $G$ is cyclic.  
Over $\Q$, there are counterexamples due to Swan, for $G=\Z/47\Z$ \cite{swan-inv}, \cite{lenstra}, \cite{plans}.
\end{exam}

\begin{exam}
\label{exam:vv}
Let $V$ be a faithful linear representation of a finite group $G$ over $k$, i.e., $G\hookrightarrow \GL(V)$. 
Then 
$$
V/G\sim_k\P(V)/G\times \P^1.
$$ 
\end{exam}

\begin{exam}
\label{exam:conn}
Let $G$ be a {\em connected} linear algebraic group and $V$ a faithful linear representation of $G$ over $k$. Then
$Y:=V/G$ admits a rational tower. Indeed, the total space of the corresponding 
(rational) fibration 
$$
V\ra V/G=Y
$$
is clearly $k$-rational, and the generic fiber is geometrically rational. 
\end{exam}

\begin{exam}
\label{exam:toric}
Toric varieties: a universal torsor of a toric variety $X_{\Sigma}$ over $k$ is given by
$$
\mathcal T_{\Sigma}=\A^n\setminus Z_{\Sigma},
$$
where $Z_{\Sigma}$ is a locally closed subvariety; we have
$$
X_{\Sigma}=\mathcal T_{\Sigma}/\!\!/T_{NS}, 
$$
where $T_{NS}$ is the N\'eron-Severi torus of $X$. Thus $X_{\Sigma}$ admits a toric rational tower. 
 \end{exam}

\begin{lemm}
Let $G$ be a finite group and $Y=V/G\in \mathcal L\mathcal Q(k)$. 
Assume that there exists an $X\in \Rat(k)$ with a generically free $G$-action, such that $X/G\in \Rat(k)$. Then 
there exists a fibration $Y_1\ra Y$ with geometrically rational generic fiber and $k$-rational $Y_1$.
\end{lemm}

\begin{proof}
We have a fibration 
$$
(X\times V)/G\ra V/G=Y,
$$ 
with generic fiber geometrically isomorphic to $X$.
On the other hand, we have a {\em vector bundle}
$$
(X\times V)/G \ra X/G
$$
with fiber $V$, since the $G$-action on $V$ is linear. Thus  $(X\times V)/G$
is $k$-rational.
\end{proof}

\begin{coro}
Let $X$ be a rational surface over an algebraically closed field $k$ 
and $G$ a linear algebraic group contained in $\Bir\Aut(X)$.  
Then $Y=V/G$ admits a rational tower.  
\end{coro}

\begin{rema}
Let $C\ra \P^1$ be a Galois cover with Galois group $G$. 
Then there exists a fibration 
$$
\pi: (C\times V)/G\ra V/G
$$
with rational total space. However, the generic fiber of $\pi$ is not necessarily geometrically rational; it is rational iff the genus $\mathsf g(C)=0$. 
\end{rema}

\section{Central extensions and wreath products}
\label{sect:wreath}

In this section, we investigate the existence of rational towers for $\mathcal L\mathcal Q(k)$, 
over algebraically closed fields $k$ of characteristic zero.  
We prove Conjecture~\ref{conj:main} for $V/G$, for special groups $G$. 

It is well-known that for finite abelian groups $A$ and their faithful linear representations $V$, the quotient $V/A$ is rational, provided the ground field $k$ contains
$n$-th roots of 1, where $n=|A|$. 
We turn to central extensions of abelian groups: 
among these, we have the free central extension 
$$
1 \ra Z\ra F^c(A)\ra A \ra 1,
$$ 
with 
$$
Z:=\wedge^2(A),
$$
generated by commutators of lifts of elements of $A$, without nontrivial relations.
The extension $F^c(A)$ unique, modulo isoclinism. 

\begin{lemm}
\label{lemm:central}
Consider a central extension of an abelian group
$$
1 \ra Z\ra G^c\ra A \ra 1.
$$
Then there exists an exact sequence
\begin{equation}
\label{eqn:Z}
1 \ra \tilde{Z}\ra  \tilde{F}^c \ra G^c\ra 1,
\end{equation}
where $\tilde{F}^c$ is isoclinic to $F^c(A)$, and $\tilde{Z}$ is abelian.
\end{lemm}

\begin{proof}
It suffices to add to $G^c$ additional elements 
which kill the image of stable cohomology 
$H^2_{st}(A,\Z/\ell)$ in $H^2(\tilde{F}^c,\Z/\ell)$ (see \cite[Section 4]{BPT} and \cite[Section 2]{BT-univer} for definitions and properties).
\end{proof}

\begin{coro}
There exist a faithful representation $W$ of $F^c=F^c(A)$ and a fibration
$$
W/F^c\ra V/G^c
$$
whose generic fiber is birational to an algebraic torus over the function field of the base. 
\end{coro}

\begin{proof}
Let $W'$ be a faithful representation of $\tilde{F}^c$ and put $W:=W'\oplus V$. We have a natural surjective homomorphism 
$W\ra V$ giving rise to a fibration
$$
W/\tilde{F}^c\ra V/G^c.
$$
Its generic fiber is geometrically isomorphic to 
$W'/\tilde{Z}$, with $\tilde{Z}$ defined in \eqref{eqn:Z}.  It suffices to recall that linear quotients of abelian  groups
are rational, when $k$ contains roots of unity. 
\end{proof}

\begin{conj}
\label{conj:fc}
Let $A$ be a finite abelian group, $F^c(A)$ its free central extension, and $V$ a faithful linear representation of $F^c$, over an algebraically closed field $k$. 
Then $V/F^c$ is stably rational. 
\end{conj}

Note that stable rationality does not change within a fixed isoclinism class, over fields containing roots of unity.

\begin{exam}
\label{exam:cyclic}
Conjecture~\ref{conj:fc} holds when $A$ is cyclic. 
When $A\simeq \Z/\ell\oplus \Z/\ell$, $F^c(A)$ is the Heisenberg group. 
The problem reduces to a {\em monomial} action, and
Conjecture \ref{conj:fc} holds as well.  Next, consider: 
$$
A:= \Z/\ell\oplus \Z/\ell\oplus \Z/\ell, \quad \text{ for } \ell=2.
$$ 
Modulo isoclinism, we can represent 
$$
F^c=F^c(A)\subset Q_{8,1}\times Q_{8,2}\times Q_{8,3},
$$
where each $Q_{8,i}$ is the group of quaternions over $\Z/2$: 
Choose a $\Z/2$-basis $\{e_1,e_2,e_3\}$ of $A$ and  
a basis $\pi_1,\pi_2,\pi_3$ of surjective homomorphisms
$$
\pi_i:A\ra \Z/2\oplus \Z/2,
$$ 
each trivial on one of the generators $e_1,e_2,$ or $e_3$.
This induces diagrams

\

\centerline{
\xymatrix{ 
1  \ar[r]    & (\Z/2)^3 \ar[r]  \ar[d]  &  F^c\ar[r] \ar[d] & (\Z/2)^3  \ar[r] \ar[d] &  1\\
 1 \ar[r]    & (\Z/2)   \ar[r]     &  Q_{8,i} \ar[r] & (\Z/2)^2 \ar[r]  & 1           
}
}

 \
 
 \noindent
 Let $V_i$ be the standard 2-dimensional representation of $Q_{8,i}$.
 Then 
 $$
 W:=V_1\oplus V_2\oplus V_3
 $$
 is a faithful representation of $F^c$. We have a $k^\times$-action on each component that commutes with the action of $F^c$. 
 It follows that
 $$
 W/F^c\sim_k      \left(  (\P^1\times \P^1\times \P^1)/  G  \right) \times (k^\times)^3,
$$  
where $G=(\Z/2)^3$ acts monomially on $k(x_1,x_2,x_3)$, as follows:
\begin{equation}
\label{eqn:sigma}
\sigma(x_i)=c_{i,\sigma} x_i^{a_{i,\sigma}}, \quad c_{i,\sigma},\, a_{i,\sigma} \in \{ \pm 1\}, \quad \forall \sigma\in G. 
\end{equation}
By \cite[Theorem 10]{yamasaki}, Type (3,3,3,1) in the notation therein,  
the field of invariants $k(x_1,x_2,x_3)^G$ is rational over $k$, when $k$ is algebraically closed. It follows that $W/F^c$ is stably rational.
\end{exam}

\begin{lemm} \cite[Lemma 2.4]{B-Petrov}
\label{lemm:wreath}
Let $G,H$ be finite groups, acting faithfully on $X$ and $Y$, respectively. 
Assume that $X/G,Y/H\in \Rat(k)$. 
Let    $K:=H\wr G$ be the wreath product, with its natural action
on $W:=X^{|H|}\times Y$. Then $W/K$ is rational. 
\end{lemm}

\begin{proof} 
Observe that the quotient  $X^{|H|}/H^{|G|}$ is a product of rational
 varieties $(X/G)^{|H|}$. The group $H$ acts on $(X/G)^{|H|}$ by permutations, and this action is
 equivalent to the linear (free permutation) action of $H$ on $V^{\oplus |H|}$, where
 $V\sim_k  X/G$.
 Thus the quotient $W/K$ is $k$-birational
 to a vector bundle over $Y/H$, 
 and hence rational.
\end{proof}

\begin{defn}
\label{defn:special-nil}
A finite solvable group $G$ is called {\em special}, if
there exists a filtration
$$
G=G_0\supsetneq G_{1} \supsetneq \cdots \supsetneq G_{r}=1
$$
such that, for all $i=0, \ldots, r-1$, one has 
\begin{itemize}
\item 
$G_i$ is a normal subgroup in $G$,
\item  
the kernel of the projection
$$
p_i:G/G_{i+1}\to G/G_{i} 
$$ 
is abelian and
\item there exists a section $s_i: G/G_{i}\to  G/G_{i+1}$
\end{itemize}
\end{defn}

\begin{prop} 
\label{prop:nil}
Let $G$ be a special solvable group and
$V$ a faithful linear representation of $G$. 
Then there exists a toric rational tower 
$$
X_n\ra \cdots \ra X_0:=V/G.
$$
\end{prop}

\begin{proof} 
By induction. We assume that the claim holds for $G_{(i)}:=G/G_{i}$. 
The group $G_{(i+1)}$
is uniquely defined by the action of  
$G_{(i)}$ on the abelian group $A_i:=G_i/G_{i+1}$.
Then there exists a surjection 
\begin{equation}
\label{eqn:mod}
 \beta_i:  B_i:=\Z/m_i  [G_i]^{r_i}\ra A_i, 
\end{equation}
for some $r_i$ and $m_i$. Its kernel $\Ker(\beta_i)$ is abelian. 

This induces a surjection from the wreath product 
$G_{(i)} \wr  (\Z/m_i)^{r_i}$ onto $G_{(i+1)}$. 
Let
$W_i$ be a faithful representation of $B_i$. Then 
there is a $G$-equivariant projection 
$$
W_i^{|G_i|}\ra W,
$$
onto some linear representation $W$ of $A_i$, corresponding to the surjection of modules   \eqref{eqn:mod}.
Let $V_i$ be a faithful representation of $G_{(i)}$ and consider the surjection 
$$
V_i \oplus W_i^{|G_{(i)}|} \ra V_i\oplus W. 
$$
Its kernel $\Ker_i$ is a linear space with a faithful action of the abelian group $\Ker(\beta_i)$. 
The quotient $\Ker/\Ker(\beta_i)$ is toric, it is the fiber of the projection
\begin{equation}
\label{eqn:tr}
(V_i \oplus W_i^{|G_{(i)}|})/W_i \ra (V_i\oplus W)/G_{(i+1)}.
\end{equation}
By induction hypothesis and Lemma~\ref{lemm:wreath}, the left side of \eqref{eqn:tr} admits a tower of toric fibrations with rational source. 
This implies the claim for $i+1$. 
\end{proof}

\begin{rema}
Our argument is similar to \cite[Theorem 3.3]{saltman-generic} that was focused on the Inverse Galois Problem. 
\end{rema}

\begin{coro}
\label{coro:cr}
Assume that $G$ is a special solvable group. 
Then there exists a rational $G$-variety $X$ with $X/G\in \Rat(k)$.
\end{coro}

\begin{proof}
By induction, as in the proof of Proposition~\ref{prop:nil}. 
It suffices to observe that $X_{i+1}$ is a quotient of a vector bundle over $X_i$ by an abelian linear action on the generic fiber. 
The corresponding quotient is rational, when $X_i$ is rational. 
\end{proof}

Thus we obtain many groups with nontrivial 
(equivalent to linear) embedding into Cremona groups, with rational quotients.

\begin{coro} 
\label{coro:prop16}
Let $V$ be a faithful representation of an $\ell$-group  $G^c$, a finite central extension of an abelian $\ell$-group $A$. 
Then $V/G^c$ admits a tower of toric fibrations with rational source. 
\end{coro}

\begin{proof} 
We apply induction on the $\ell$-rank of $A$. The claim is trivial when rank equals 1. 
Let $A':=A\oplus \Z/\ell^m$. Let $F^c:=F^c(A')$ be the 
free central extension of $A'$. 
We have a surjection
$$
F^c(A')\twoheadrightarrow F^a(A),
$$
with a canonical section and abelian kernel.  Now we apply Proposition~\ref{prop:nil}.
\end{proof}

As mentioned in Section~\ref{sect:background}, 
Noether's problem, i.e., the rationality of $V/G$, has 
a negative solution. However, one can consider more general, 
nonlinear, generically free $G$-actions on rational varieties.
A version of Conjecture~\ref{conj:main}, and of Noether's problem, is the following

\begin{conj}
\label{conj:main2}
Let $G$ be a finite group. Then there exists a $k$-rational $G$-variety $X$ with generically free $G$-action such that $X/G$ is $k$-rational.
\end{conj}

\section{Group theory}
\label{sect:group}

In this section we consider finite simple groups $G$. 
Our main observation is that the Sylow subgroups $\Syl_{\ell}(G)$ of most 
simple groups $G$ satisfy the assumptions of Proposition~\ref{prop:nil}; the corresponding $V/\Syl_{\ell}(G)$ admit 
a rational tower.  Below we sketch a proof in a special case. 

\

Let $G=\PGL_{n}(\F_q)$, with $q=p^m$, for $m\in\N$.  
 \begin{enumerate}
 \item 
 Consider $\Syl_p(G)$. It is conjugate to a subgroup of upper-triangular matrices $U_n\subset G$. 
There is  a natural projection
 $$
 U_{r}\to U_{r-1}, \quad \forall r, 
 $$
 with a section $s_{r-1}$ and abelian kernel, as in the assumptions of Proposition~\ref{prop:nil}, 
 which induces the corresponding structure on $\Syl_p(G)$. 
  \item Consider $\Syl_{\ell}(G)$, with $\ell\neq p$. Every such subgroup 
  is a subgroup of the normalizer  $N(T)$ of a (possibly nonsplit) maximal torus $T\subset G$. 
We have an extension
  $$
  1\ra T\ra N(T)\ra W_T(G)\ra 1.
  $$
Note that $\Syl_{\ell}(G)$ admits a projection onto the corresponding $\Syl_{\ell}(W_T(G))$ and the induced extension, by an abelian $\ell$-group,
splits (for $\ell\neq 2$). Since 
Proposition~\ref{prop:nil} holds  for 
$$
\Syl_{\ell}(W_T(G))=\Syl_{\ell}(\mathfrak S_{n-1})
$$ 
(it is an iterated wreath product extension by cyclic $\ell$-groups), 
it also holds for $\Syl_{\ell}(G)$. 
\end{enumerate}
Similar arguments apply to other finite groups of Lie type. 
Additional considerations are needed for some sporadic simple groups, and small primes.

\section{Actions in small dimensions}
\label{sect:special}

In this section, we survey related results on rationality of quotients $X/G$, for low-dimensional rational varieties $X$ over (possibly nonclosed) fields  $k$ 
of characteristic zero.    

Let $V$ be a faithful linear representation of a finite group $G$ over $k$. 
When  $G$ is abelian, $V/G$ is rational; however, $\P(V)/G$ need not be $k$-rational, even when $\dim(V)=4$ \cite[Example 2.3]{ahmad}.

When $\dim(V)\le 3$, $V/G$ is rational over algebraically closed $k$: indeed, by Example~\ref{exam:vv}, 
it suffices to consider the unirational surface $\P(V)/G$ which is rational. 
However, the situation is different over nonclosed fields.

\

{\em Dimension 2.} 
There is a large literature on rational $G$-surfaces $X$ over nonclosed fields, i.e., 
actions of finite groups $G\subset \Aut(X)$ (see \cite{manin}). 
We have $X/G\sim_k S/G$,
where $S$ is a $G$-minimal del Pezzo surface or a conic bundle. 
If $S$ is a conic bundle then 
$G\subset \PGL_2(\bar{k})$ (see \cite[Theorem 1.3]{trepalin}). 
We start with several rationality results: let $X$ be a smooth del Pezzo surface with $X(k)\neq \emptyset$ and $G\subset \Aut(X)$. 
\begin{itemize}
\item 
If $\deg(X)=K_X^2\ge 5$ then $X/G \in \Rat(k)$ \cite[Corollary 1.4]{trepalin-0}, \cite[Theorem 1.1]{trepalin-1}.
\item
If $\deg(X)=3$ and $|G|\neq 3$ then $X/G\in \Rat(k)$ \cite[Theorem 1.3]{trepalin-1}.  
\end{itemize}

\begin{exam}\cite[Section 5]{trepalin-1}
Let $X\subset \P^3$ be a smooth cubic surface given by
$$
f_3(x,y) +zt(ux+vy)+z^3+wt^3=0,
$$
where $f_3$ is a form of degree 3,  $(x:y:z:t)$ are coordinates in $\P^3$, and $u,v,w$ are parameters. Assume that the Galois group 
of $f_3$ is $\Z/2$. Then $X$ admits an action of $G=\Z/3$, and $X/G$ is $k$-birational to a nonrational over $k$, 
minimal degree 4 del Pezzo surface $S$, admitting a conic bundle $S\ra \P^1$ over $k$.  By \cite{CT-Sko}, $X/G$ admits a rational tower: 
a universal torsor is $k$-rational. 
\end{exam}

On the other hand, over nonclosed $k$, the set of 
$k$-birational types of quotients of conic bundles over $\P^1$ may be infinite \cite[Theorem 1.8]{trepalin}. 
E.g., this holds when $G=\mathfrak A_5$ and 
\begin{itemize}
\item not every element of $k$ is a square,
\item $\sqrt{-1}, \sqrt{5}\in k$.
\end{itemize}

\begin{prob}
Establish the existence of rational towers for Del Pezzo surfaces over nonclosed fields.
\end{prob}

\

{\em Dimension 3.}  Over algebraically closed $k$, stable rationality of quotients $\P^3/G$ is unknown for 
\begin{itemize}
\item central $\Z/2$-extensions of the following groups:
$$
\tilde{\mathfrak S}_5, \tilde{\mathfrak A}_6, \tilde{\mathfrak S}_6, \tilde{\mathfrak A}_7, \SL_2(\mathbb F_7),
$$
\item extensions of $\mathfrak A_5$, respectively, $\mathfrak A_6$ by a group $N$ of order 64.  
\end{itemize}
For all other subgroups $G\subset \GL_4(k)$, $\P^3/G$ is stably rational \cite{prokhorov}. Rationality over nonclosed fields has not been addressed.

\bibliographystyle{alpha}
\bibliography{fibrations}

\end{document}